\documentclass{amsart}
\date{\today}
\usepackage{graphicx, color, enumerate}
\usepackage{amsmath, amssymb, amscd, mathrsfs}
\title[Fold singularities on spacelike CMC surfaces]{%
Fold singularities on spacelike CMC surfaces in Lorentz-Minkowski space}
\author[A.~Honda]{Atsufumi Honda}
\address{%
   National Institute of Technology, Miyakonojo College, 
   Yoshio, Miyakonojo 885-8567, Japan.
}
\email{atsufumi@cc.miyakonojo-nct.ac.jp}
\author[M.~Koiso]{Miyuki Koiso}
\address{%
   Institute of Mathematics for Industry, Kyushu University, 
   744, Motooka, Nishi-ku, Fukuoka 819-0395, Japan.
}
\email{koiso@math.kyushu-u.ac.jp}
\author[Saji~K.]{Kentaro Saji}
\address{%
   Department of Mathematics, Faculty of Science, Kobe University, 
   Rokko, Kobe 657-8501, Japan.
}
\email{saji@math.kobe-u.ac.jp}
\thanks{(Corresponding author) A.\ Honda;
{\it E-mail address}:
{\tt atsufumi@cc.miyakonojo-nct.ac.jp.}
{\it Address}:
National Institute of Technology, Miyakonojo College,  
Yoshio, Miyakonojo 885-8567, Japan.
{\it Current address}:
Vienna University of Technology,
Wiedner Hauptstra{\ss}e 8-10/104,
A-1040 Vienna, Austria.
}
\usepackage{amsthm}
\theoremstyle{plain}
 \newtheorem{theorem}{Theorem}[section]

 \newtheorem{step}{Step}

 \newtheorem{fact}[theorem]{Fact}
 \newtheorem*{fact*}{Fact}
 \newtheorem{lemma}[theorem]{Lemma}
 
 \theoremstyle{remark}
 \newtheorem{definition}[theorem]{Definition}
 \newtheorem{remark}[theorem]{Remark}
 \newtheorem*{acknowledgements}{Acknowledgements}
 
\numberwithin{equation}{section}

\newcommand{\R}{\boldsymbol{R}}
\newcommand{\C}{\boldsymbol{C}}
\newcommand{\D}{\boldsymbol{D}}

\newcommand{\inner}[2]{\left\langle{#1},{#2}\right\rangle}
\newcommand{\vect}[1]{\boldsymbol{#1}}

\renewcommand{\Re}{\operatorname{Re}}
\renewcommand{\Im}{\operatorname{Im}}

\begin{document}
\begin{abstract}
Fold singular points play important roles in the theory of maximal surfaces.
For example, if a maximal surface admits fold singular points,
it can be extended to a timelike minimal surface analytically.
Moreover, there is a duality between conelike singular points and folds.
In this paper, we investigate fold singular points
on spacelike surfaces with non-zero constant mean curvature
(spacelike CMC surfaces).
We prove that spacelike CMC surfaces do not admit fold singular points.
Moreover, we show that 
the singular point set of any conjugate CMC surface 
of a spacelike Delaunay surface 
with conelike singular points
consists of $(2,5)$-cuspidal edges.
\end{abstract}
\maketitle

\section{Introduction}

An immersed surface in the Lorentz-Minkowski $3$-space $L^3$
is called of {\it zero mean curvature\/}
if it is locally a graph $x_0=f(x_1,\,x_2)$ satisfying
\[
  (1-f_{x_2}^2)f_{x_1 x_1} + 2 f_{x_1}f_{x_2}f_{x_1 x_2} + (1-f_{x_1}^2)f_{x_2 x_2}=0
\]
or a plane parallel to $x_0$-axis,
where we regard $L^3$ as an affine space 
$\R^3=\{(x_0,x_1,x_2)\}$ with the Lorentz metric of signature $(-,+,+)$,
and denote $f_{x_1}=\partial f/\partial x_{1}$, and so on.
At the point satisfying $1-f_{x_1}^2-f_{x_2}^2>0$ 
(resp.\ $1-f_{x_1}^2-f_{x_2}^2<0$),
the zero mean curvature surface is spacelike maximal (resp.\ timelike minimal). 
Although any complete maximal surface in $L^3$ is a spacelike plane \cite{Calabi},
there are nontrivial zero mean curvature surfaces of mixed type 
(\cite{Kobayashi_Tokyo}, \cite{Gu1985}, 
\cite{TkachevSergienko}, \cite{Klyachin2003}, \cite{KKSY}, 
\cite{FKKRSTUYY}, \cite{FRUYY2014}, 
\cite{FKKRSUYY_Osaka}, \cite{FKKRSUYY_Okayama}, 
\cite{FKKRUY_JM} and \cite{FKKRUY_entire}),
where a (connected) surface in $L^3$ is called \emph{of mixed type} if
its spacelike and timelike parts are both non-empty.

\begin{figure}[htb]
\begin{center}
 \begin{tabular}{{c@{\hspace{20mm}}c}}
  \resizebox{5cm}{!}{\includegraphics{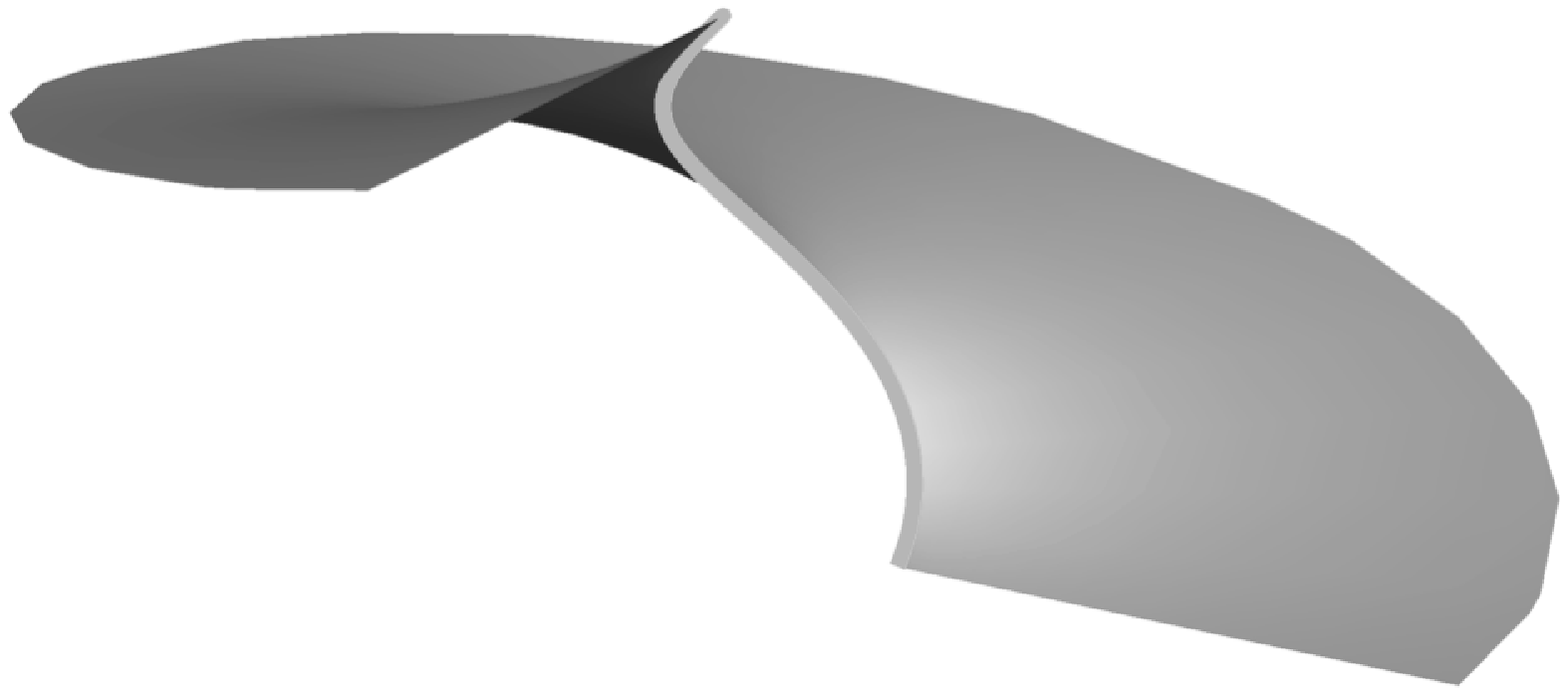}}&
  \resizebox{5cm}{!}{\includegraphics{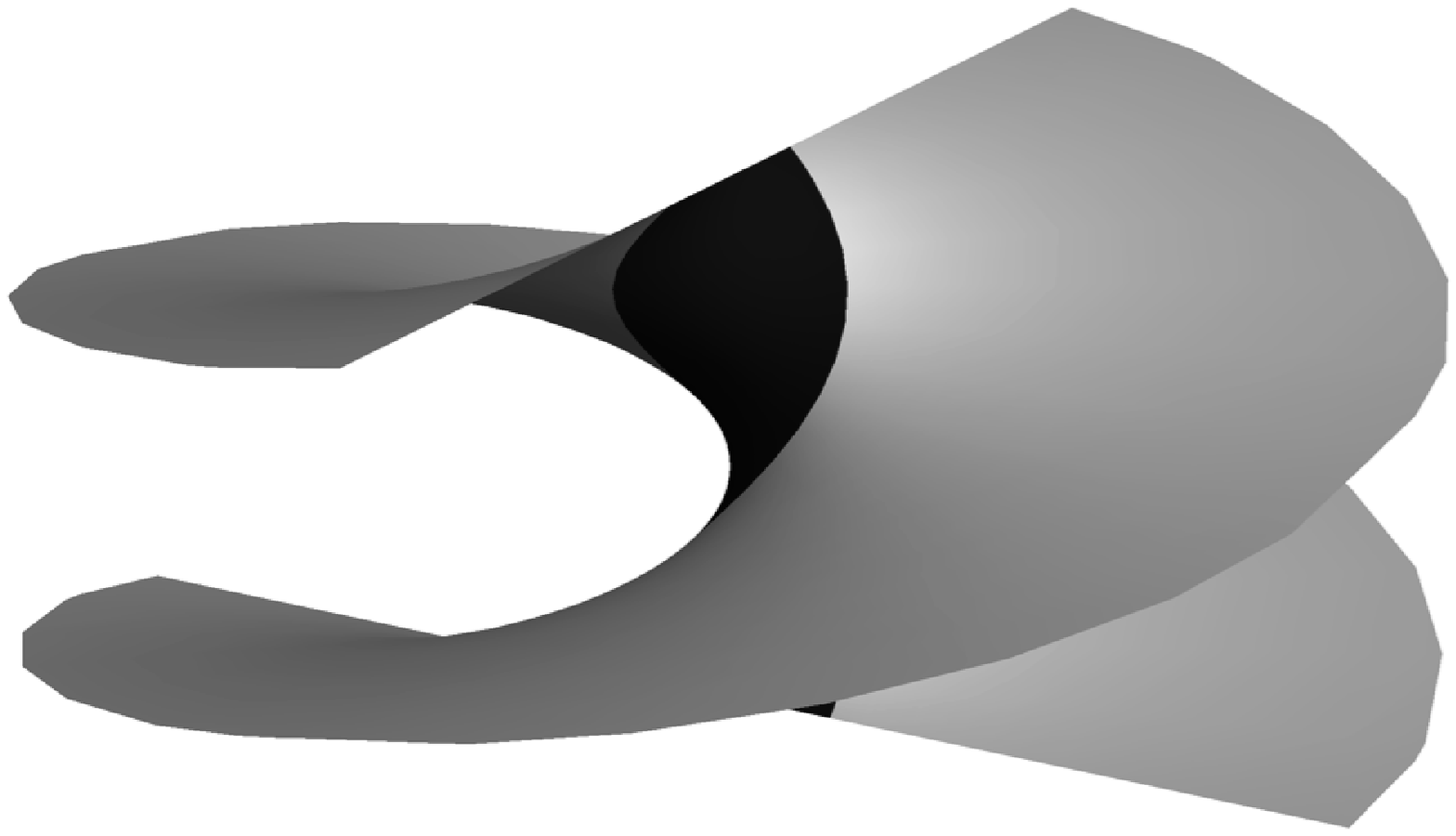}} \\
  {\footnotesize  Maximal helicoid.} &
  {\footnotesize  Zero mean curvature helicoid.}
 \end{tabular}
 \caption{The left figure is the spacelike maximal helicoid 
 whose singular point set consists of fold singular points.
 The right figure is the zero mean curvature helicoid of mixed type 
 which is an extension of the left maximal helicoid.
 }
\label{fig:extension}
\end{center}
\end{figure}

According to Gu \cite{Gu1985, Gu1987, Gu1990},
Klyachin \cite{Klyachin2003} and Kim-Koh-Shin-Yang \cite{KKSY},
on a neighborhood of a non-degenerate type-changing point 
of a zero mean curvature surface of mixed type,
its spacelike part is a {\it maxface\/} with {\it fold\/} singular points 
\cite{FKKRSUYY_Okayama} (cf.\ Definition \ref{def:fold}),
where a `maxface' is a maximal surface with admissible singular points
introduced by Umehara-Yamada \cite{UY_maxface}.
Conversely, a maxface with fold singular points can be extended analytically
to a zero mean curvature surface which changes causal type.
For the definition of non-degenerate type-changing points,
see \cite{Gu1987, Gu1990} and \cite{FKKRSUYY_Okayama}.
Roughly speaking, there is a one-to-one correspondence
between fold singular points and zero mean curvature surfaces of mixed type.

In this paper, we consider fold singular points
on non-maximal spacelike surfaces of constant mean curvature
(i.e., spacelike CMC surfaces).
Spacelike CMC surfaces have a significant importance in physics \cite{MarsdenTipler}.
Umeda \cite{Umeda} introduced a class of 
spacelike CMC surfaces with admissible singularities 
called `{\it generalized spacelike CMC surfaces\/}'
(cf.\ Definition \ref{def:geneCMC}),
and investigated their singularities.
On the other hand, Brander defined and investigated 
spacelike CMC surfaces with singularities
using the DPW method \cite{BranderCMC}.
Although Umeda \cite{Umeda} and Brander \cite{BranderCMC} 
exhibited various examples 
of spacelike CMC surfaces with singularities
(such as cuspidal edges, swallowtails, cuspidal cross caps and
conelike singular points),
spacelike CMC surfaces with {\it fold\/} singular points were not known.
Here, we show the following:

\begin{theorem}\label{thm:fold}
Generalized spacelike CMC surfaces 
do not admit any fold singular points.
\end{theorem}

By this theorem, we can not expect 
the existence of CMC surfaces of mixed type.
In fact, in \cite{HKKUY} with Kokubu, Umehara and Yamada, 
the first and second authors have proved that 
there do not exist (connected) CMC surfaces of mixed type.

On the other hand, it is known that 
for a maxface with conelike singular points,
its conjugate has fold singular points,
and vice versa (\cite{KimYang}, \cite{FKKRSUYY_Okayama}).

\begin{figure}[htb]
\begin{center}
 \begin{tabular}{{c@{\hspace{10mm}}c}}
  \resizebox{5cm}{!}{\includegraphics{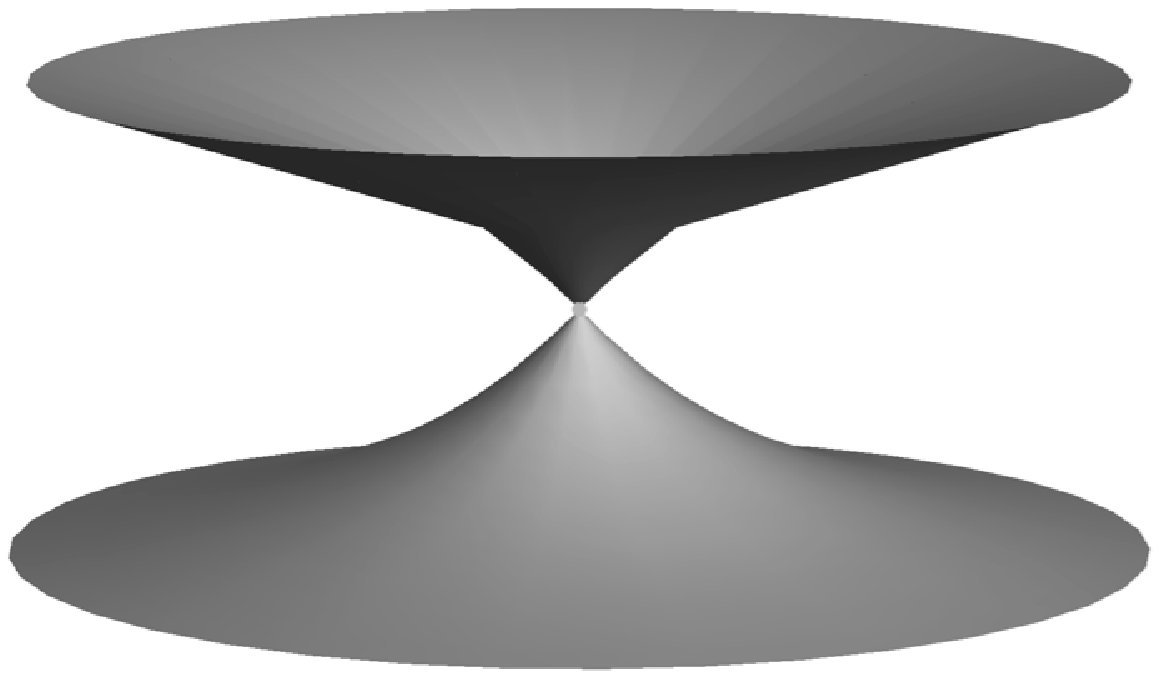}}&
  \resizebox{5cm}{!}{\includegraphics{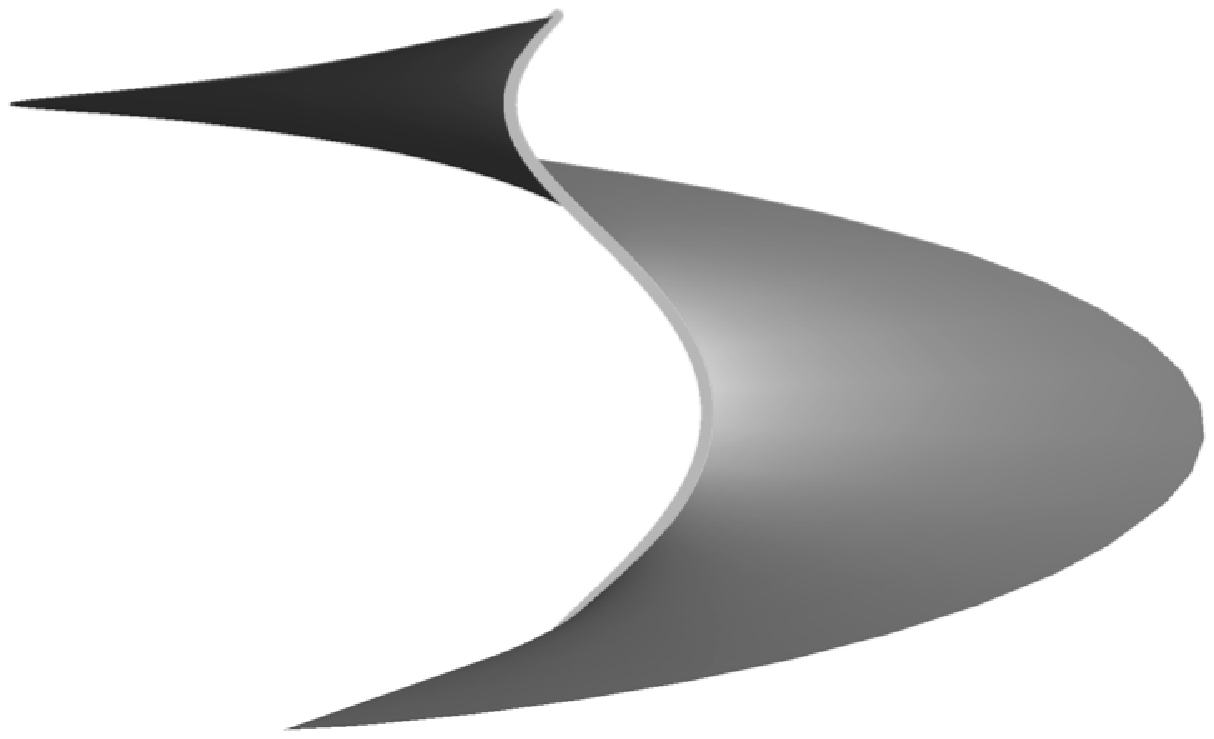}} \\
  {\footnotesize  Maximal catenoid.} &
  {\footnotesize  Maximal helicoid.}
 \end{tabular}
 \caption{The left figure is a maximal catenoid which is a 
        maxface with conelike singular points.
        The right figure is a maximal helicoid 
        which is the conjugate of a maximal catenoid.
        Its singular point set consists of fold singular points.}
\end{center}
\end{figure}

\begin{figure}[htb]
\begin{center}
 \begin{tabular}{{c@{\hspace{10mm}}c}}
  \resizebox{5cm}{!}{\includegraphics{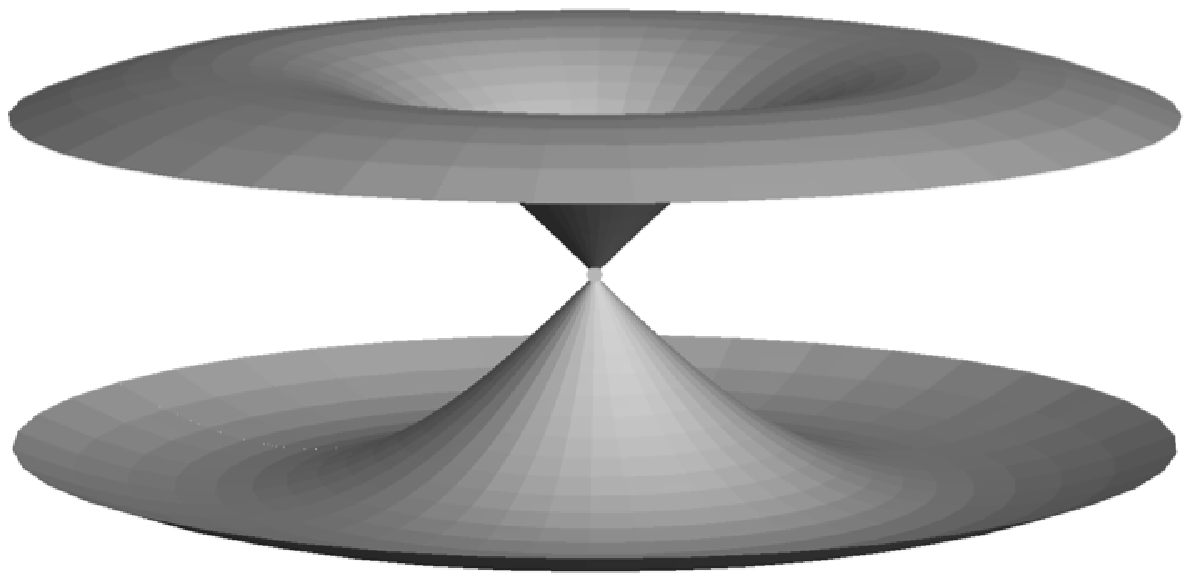}}&
  \resizebox{3cm}{!}{\includegraphics{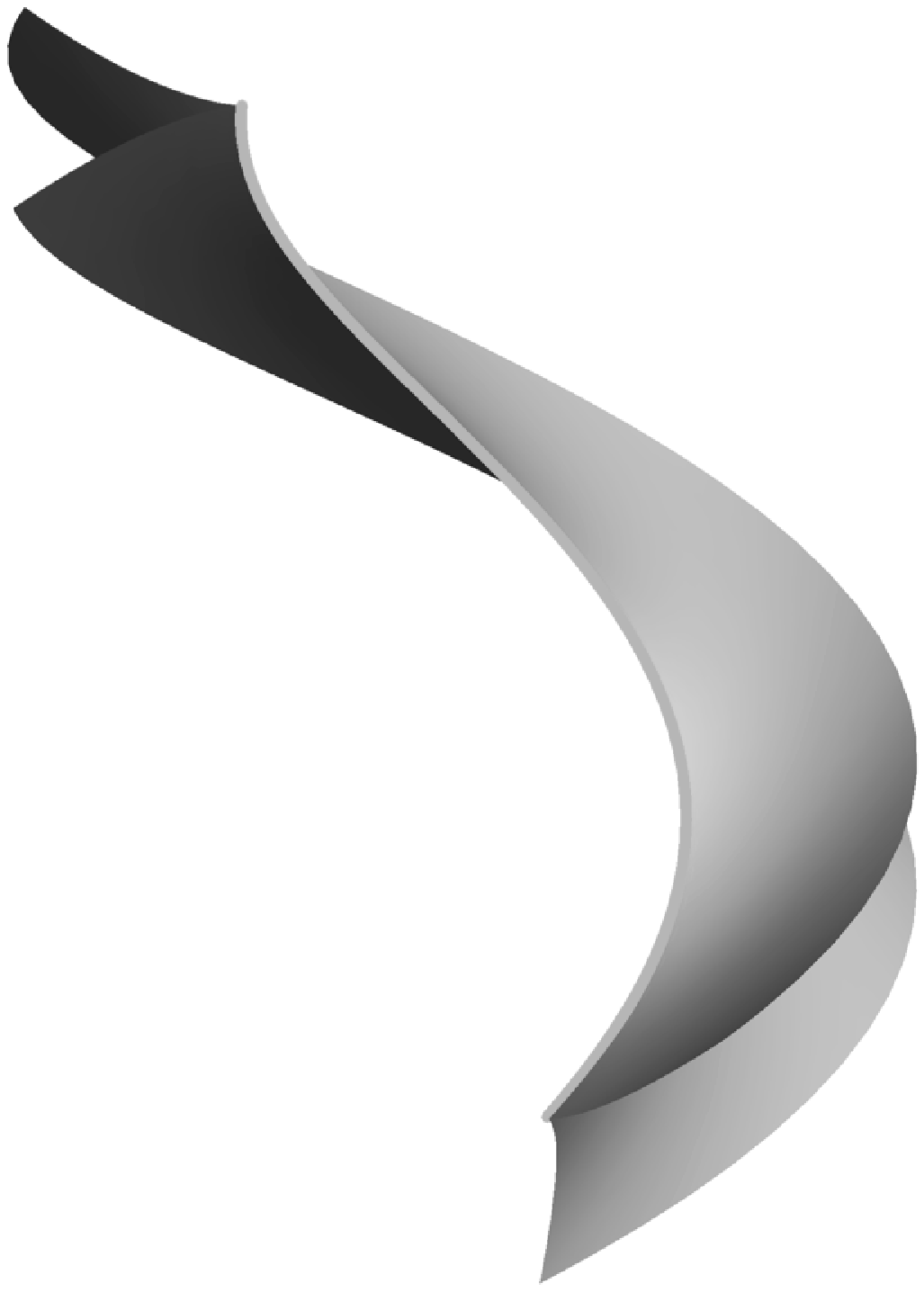}} \\
  {\footnotesize  A spacelike Delaunay surface.} &
  {\footnotesize  The conjugate.}
 \end{tabular}
 \caption{The left figure is a spacelike Delaunay surface whose axis is timelike.
        Its singular point set consists of conelike singular points.
        The right figure is the conjugate of the left. Its singular points are {\it not\/} fold.}
        \label{fig:CMC-duality}
\end{center}
\end{figure}

Although fold singular points never appear
on generalized spacelike CMC surfaces by Theorem \ref{thm:fold},
there exist generalized spacelike CMC surfaces
having conelike singular points 
(cf.\ Figure \ref{fig:CMC-duality}, Remark \ref{rem:conelike}).
Therefore, it is natural to ask as follows: 
{\it What are the singular points which appear 
on the conjugate of generalized spacelike CMC surfaces 
with conelike singular points ?\/}
We answer to this problem in the case of 
generalized spacelike CMC surfaces of revolution
(i.e., spacelike {\it Delaunay\/} surfaces).

\begin{theorem}\label{thm:2-5}
For a spacelike Delaunay surface with conelike singular points,
its conjugate has $(2,5)$-cuspidal edges.
\end{theorem}

We remark that maxfaces do {\it not\/} admit any $(2,5)$-cuspidal edges 
(cf.\ Remark \ref{rem:maxface_25}).
By Theorem \ref{thm:fold} and Theorem \ref{thm:2-5}, 
we may conclude that 
the singularity types of spacelike CMC surfaces 
are different from those of maximal surfaces.

For the proof of Theorem \ref{thm:2-5}, we give a criterion for $(2,5)$-cuspidal edges 
in Theorem \ref{thm:criterion}.
A criterion for $(2,5)$-cusps of plane curves 
can be found in \cite[Theorem 1.23]{Porteous}.
Recently, Ishikawa-Yamashita found an interesting
example \cite[Example 12.4]{IshiYama2015} 
(\cite[Example 9.4]{IshiYama2016}) of
tangent surface having $(2,5)$-cuspidal edges
in non-projectively-flat $3$-space.

This paper is organized as follows.
In Section \ref{sec:prelim}, we recall some basic facts on 
spacelike CMC surfaces and surfaces with singularities.
In Section \ref{sec:proof-main}, we review 
generalized spacelike CMC surfaces and prove Theorem \ref{thm:fold}.
In Section \ref{sec:criterion}, we give a criterion for $(2,5)$-cuspidal edges 
(cf.\ Theorem \ref{thm:criterion}) and show Theorem \ref{thm:2-5}.

\section{Preliminaries}\label{sec:prelim}

Denote by $L^3=(\R^3\,,\,\inner{~}{~})$ 
the Lorentz-Minkowski $3$-space
with the Lorentzian inner product
$\inner{\vect{x}}{\vect{x}}=-x_0^2+x_1^2+x_2^2$, 
where $\vect{x}=(x_0,x_1,x_2) \in L^3$.
We set $H^2$ and $H^2_{\pm}$ as
\[
  H^2 := \{ p \in L^3 \,;\, \inner{p}{p}=-1 \}
\]
and $H^2_{\pm} := H^2\cap L^3_{\pm}$, respectively,
where $L^3_{\pm}:=\{(x_0,x_1,x_2)\in L^3 \,;\, x_0\gtrless 0 \}$.
That is, $H^2$ is the union of two hyperbolic planes.
The stereographic projection $\pi : H^2\rightarrow \hat\C$ is defined by
\begin{equation}\label{eq:stereo-proj}
  \pi(p)=\frac{x_1+i x_2}{1-x_0},\qquad
  p=(x_0,\,x_1,\,x_2)\in H^2,
\end{equation}
where $\hat\C := \C\cup\{\infty\}$ is the Riemann sphere.
If we denote by $\D=\{ z \in \C\, ; \, |z|<1\}$ the unit disk,
each of the restrictions 
\[
  \pi|_{H^2_-} : H^2_-\rightarrow \D,\qquad
  \pi|_{H^2_+} : H^2_+\rightarrow \hat\C\setminus \bar{\D}
\]
gives diffeomorphisms onto the image,
and hence $\pi(H^2)=\hat\C\setminus S^1$,
where $\bar{\D}=\D\cup S^1$ and $S^1=\{ z \in \hat\C\, ; \, |z|=1\}$.
Therefore $\hat\C$ may be considered as 
a compactification of $H^2$.

We denote by $\Sigma$ an oriented smooth $2$-manifold.
In this paper, a {\it surface\/} in $L^3$ is defined to be an immersion $X$ 
of $\Sigma$ into $L^3$.
We denote by $ds^2=X^{\ast}\inner{~}{~}$ {\it the first fundamental form\/} of $X$.
If $ds^2$ defines a Riemannian metric on $\Sigma$,
$X$ is called {\it spacelike\/}.
Taking a (timelike) unit normal vector field 
$\nu : \Sigma \rightarrow H^2 \subset L^3$ along $X$, 
the {\it second fundamental form\/} of $X$
is given by $I\!I=\inner{-d\nu}{df}$.
Then, the {\it mean curvature function\/} $H$ is defined by
$
  H=(\kappa_1+\kappa_2)/2,
$
where $\kappa_1$, $\kappa_2$ are the {\it principal curvatures\/} of $X$.
We call the composition $g:=\pi\circ \nu : \Sigma \rightarrow \hat\C$ 
the {\it Gauss map\/} of $X$, where $\pi$ is the stereographic projection 
given by \eqref{eq:stereo-proj}.

If we take a conformal coordinate system $(U,\,z=u+i v)$
of the Riemann surface $(\Sigma,\, ds^2)$,
we may write $ds^2$ and $I\!I$ as
\[
  ds^2
  =e^{2\sigma}dz\,d\bar{z},\qquad
  I\!I=Q + \bar{Q} + H ds^2,
\]
where $Q=q\,dz^2 ~(q=\inner{f_{zz}}{\nu})$ is the {\it Hopf differential\/} of $X$.
Then, the Gauss and Codazzi equations are given by
\begin{equation}
  \label{eq:compatiblity}
  4\sigma_{z\bar{z}}=e^{2\sigma}H^2-4e^{-2\sigma}\left|q\right|^2,\qquad
  q_{\bar{z}}=e^{2\sigma} H_{z},
\end{equation}
respectively.
According to the fundamental theorem of surface theory, 
if the triplet $(ds^2, H, Q=q\,dz^2)$ 
defined on a simply connected domain $U \subset \C$
satisfies \eqref{eq:compatiblity},
there exists a conformal immersion $X : U \rightarrow L^3$
such that $ds^2$ is the first fundamental form, 
$Q$ is the Hopf differential, 
and $H$ is the mean curvature 
of $X$.

\subsection{Spacelike CMC surface, associate family, Kenmotsu-type representation formula}

A spacelike surface in $L^3$ is said to be {\it CMC-$H$\/} or {\it CMC\/}, 
if its mean curvature is identically a constant $H$.
In particular, a surface is called {\it maximal\/}
if its mean curvature is identically zero.
Let $X : \Sigma \rightarrow L^3$ be a spacelike CMC-$H$ surface
of which the triplet is given by $(ds^2,\,H,\, Q)$.
If $\Sigma$ is simply connected, there exists a family of spacelike CMC-$H$ surfaces 
$\{ X_\theta \}_{\theta\in S^1}=\{ X_\theta : \Sigma \rightarrow L^3 \}_{\theta\in S^1}$,
where the triplet of $X_\theta$ is given by $(ds^2,\,H,\, Q_\theta=e^{i\theta} Q)$
for each $\theta\in S^1$.
The family $\{ X_\theta \}_{\theta\in S^1}$ is called the {\it associate family\/} of $X$.
We call $X^\#:=X_{\pi/2}$ the {\it conjugate\/} of $X$.

A smooth map 
$g : \Sigma \rightarrow \pi(H^2)$ defined on a Riemann surface $\Sigma$
is called {\it harmonic\/} if it satisfies
\begin{equation}\label{eq:HME}
  g_{z\bar{z}}+\frac{2\bar{g}}{1-|g|^2}g_{z}g_{\bar{z}}=0,
\end{equation}
where $z$ is a local conformal coordinate of $\Sigma$.
Akutagawa-Nishikawa \cite{AkutagawaNishikawa} proved 
the Kenmotsu-type representation formula 
for spacelike CMC surfaces as follows.

\begin{fact}[\cite{AkutagawaNishikawa}]
Let $X : \Sigma \rightarrow L^3$ be a conformal non-maximal CMC-$H$ immersion
defined on a simply connected Riemann surface $\Sigma$.
Then there exists a harmonic map $g=g(z)$ such that
\begin{equation}\label{eq:AN-rep}
  X(z)=\frac{2}{H}\Re \int^{z}_{z_0} 
  \left(-2g, 1+g^2, i(1-g^2) \right)\omega,
\end{equation}
where $z_0\in \Sigma$ is a base point and
\begin{equation}\label{eq:omega}
  \omega:=\hat\omega dz\qquad
  \left(\hat\omega = \dfrac{\bar{g}_z}{(1-|g|^2)^2}\right).
\end{equation}
Conversely, 
take a non-holomorphic harmonic map $g : \Sigma \rightarrow \pi(H^2)$
defined on a simply connected Riemann surface $\Sigma$
and a base point $z_0\in \Sigma$.
Then the integration in \eqref{eq:AN-rep} does not
depend on the choice of a path joining $z_0$ and $z$,
and $X$ in \eqref{eq:AN-rep} is a spacelike CMC-$H$ immersion 
whose Gauss map is $g$.
Furthermore, the first and second fundamental forms are given by
\[
  ds^2=(1-|g|^2)^2|\omega|^2,\qquad
  I\!I=Q+\bar{Q}+H\,ds^2\qquad
  \left(Q=-\omega dg \right),
\]
respectively.
\end{fact}

\subsection{Surface with singularities}

For a smooth map $X : \Sigma \rightarrow \R^3$
of a smooth $2$-manifold $\Sigma$,
a point $p\in \Sigma$ is called {\it singular\/},
if $X$ is not immersion at $p$.
A non-singular point is called {\it regular\/}.
We denote by $S(X)$ (resp.\ $R(X)$) the set of singular (resp.\ regular) points of $X$.

A smooth map 
$X : \Sigma \rightarrow \R^3$
is called a {\it frontal\/}, if for any point $p\in \Sigma$,
there exist an open neighborhood $U$ of $p$
and a smooth map $\vect{n} : U \rightarrow S^2$
such that $dX(\vect{v})\cdot \vect{n}(q)=0$ holds
for each $q\in U$ and $\vect{v}\in T_q\Sigma$,
where the dot ``\,$\cdot$\,'' means the Euclidean inner product.
Such a map $\vect{n}$ is called a ({\it Euclidean\/}) {\it unit normal vector field\/} along $X$.
(If $(X,\vect{n})$ is an immersion, $X$ is called a {\it front\/}.)
We call $\lambda := \det (X_u,\,X_v,\,\vect{n})$ 
the {\it signed area density function\/}, 
where $(u,v)$ is the coordinates of $U$.
A point $p\in U$ is a singular point of $X$
if and only if $\lambda(p)=0$.
If $d\lambda(p)\neq 0$, 
a singular point $p$ is called {\it non-degenerate\/}.
We remark that if $p$ is non-degenerate,
then ${\rm rank}(dX)_p=1$ holds.
By the implicit function theorem,
there exists a regular curve $\gamma(t)$
($|t|<\varepsilon$)
on the $uv$-plane such that $\gamma(0)=p$
and the image of $\gamma$
coincides with the singular point set $S(X)$ near $p$,
where $\varepsilon>0$.
We call $\gamma(t)$ the {\it singular curve\/}
and $\gamma'=d\gamma/dt$ the singular direction.
Then, there exists a non-zero smooth vector field $\eta(t)$
along $\gamma(t)$ such that $\eta(t)$
is a null vector (i.e., $dX(\eta(t))=0$) for each $t$.
Such a vector field $\eta(t)$ is called a {\it null vector field\/}.
If $\gamma'(0)$ is not proportional to $\eta(0)$,
then $p=\gamma(0)$ is called of the {\it first kind\/}.
In this setting, we can extend $\xi(t):=\gamma'(t)$ and $\eta(t)$ 
to smooth vector fields $\xi=\xi(u,v)$ and $\eta=\eta(u,v)$ on $U$, respectively. 

The following two lemmas are well-known (see \cite{GG}).
They play crucial roles in Whitney \cite{Whitney} to give 
a criterion for a given smooth map to be a cross cap. 
Let $h(u,v)$ be a smooth function defined around the
origin.

\begin{fact}[Division Lemma]
\label{fact:division}
 If $h(u,0)$ vanishes for sufficiently small $u$,
 then there exists a smooth function 
 $\tilde{h}(u,v)$ defined around the origin such that 
 $h(u,v)=v\tilde h(u,v)$ holds.
\end{fact}

\begin{fact}[Whitney Lemma]
\label{fact:whitney}
 If $h(u,v)=h(-u,v)$ holds for sufficiently small $(u,v)$,
 then there exists a smooth function 
$\tilde h(u,v)$ defined around the
 origin such that $h(u,v)=\tilde h(u^2,v)$ holds.
\end{fact}

Using Fact \ref{fact:division}, \ref{fact:whitney},
the following facts can be proved.
(See \cite[Theorem 1.4]{FSUY}.)

\begin{fact}
\label{fact:splitting}
For a smooth function $h(u,v)$,
there exist smooth functions
$\alpha(u,v)$ and $\beta(u,v)$
defined around the origin 
such that
$
  h(u,v)=\alpha(u,v^2)+v\beta(u,v^2)
$
holds.
\end{fact}

\begin{fact}
\label{fact:sing-form}
Let $p=(0,0)$
be a non-degenerate singular point
of a frontal
$X:U\rightarrow \R^2$ defined on a domain $U\subset \R^2$,
$\gamma(t)$ $(|t|<\varepsilon)$
the singular curve  passing through $p=\gamma(0)$,
and $\eta(t)$ a null vector field along $\gamma(t)$.
If $p$ is of the first kind,
then there exist diffeomorphisms
$\Phi$ defined on $\R^3$,
$\varphi$ on $\R^2$,
and a smooth function $h(u,v)$
defined around the origin
such that 
$
  (\Phi\circ X \circ \varphi) (u,v) = (u,v^2,v^3 h (u,v))
$
holds.
\end{fact}

\section{Generalized spacelike CMC surface and Fold singularity}
\label{sec:proof-main}

In this section, 
first we review some definitions and introduce
some properties on singularities of generalized spacelike CMC surfaces.
Then we shall prove Theorem \ref{thm:fold}.

\subsection{Generalized spacelike CMC surface}

Umeda \cite{Umeda}
investigated singularities of 
spacelike CMC surfaces with admissible singularities
called {\it generalized spacelike CMC surfaces\/}.

\begin{definition}[{\cite{Umeda}}]\label{def:geneCMC}
Let $\Sigma$ be a Riemann surface.
For a smooth map 
$g : \Sigma \rightarrow \hat\C$,
set $S_1(g):=\{ p\in \Sigma \,;\, |g(p)|=1 \}$
and $\omega$ as in \eqref{eq:omega}.
\begin{itemize}
\item
A smooth map $g : \Sigma \rightarrow \hat\C$ 
is called {\it regular extended harmonic map\/} if
the following two conditions hold:
   \begin{itemize}
   \item[$(1)$] $\omega$ can be extended to a $1$-form of class $C^1$
         across $S_1(g)$,
   \item[$(2)$] $g$ satisfies $g_{z\bar{z}}+2(1-|g|^2)\bar{g}g_z \bar{\hat{\omega}}=0$,
   where $\omega=\hat{\omega}(z)\, dz$.
   \end{itemize} 
\item
For a non-holomorphic regular extended harmonic map
$g : \Sigma\rightarrow \hat\C$
and a non-zero constant $H\neq0$,
if the map $X : \Sigma \rightarrow L^3$
given by \eqref{eq:AN-rep} is well-defined,
then $X$ is called a {\it generalized spacelike constant mean curvature\/} surface
$($or {\it CMC, CMC-$H$\/}$)$.
\end{itemize}
The map $g$ is called the {\it Gauss map\/} of $X$.
\end{definition}

By the condition $(1)$ in Definition \ref{def:geneCMC},
we have
\begin{equation}\label{eq:bar-z}
  g_{\bar{z}}(p)=0  \qquad \text{for} 
  \quad p\in S_1(g).
\end{equation}

\begin{fact}[{\cite[Proposition 3.8]{Umeda}}]
\label{fact:geneCMC-sing}
Let $X : \Sigma \rightarrow L^3$
be a generalized spacelike CMC surface with the Gauss map $g$.
Set $S_{\infty}(g):=\{ p\in \Sigma \,;\, |g(p)|=\infty \}$.
Then,
\begin{itemize}
\item[(i)]
a point $p\in \Sigma\setminus S_{\infty}(g)$
is a singular point of $X$ 
if and only if $|g(p)|=1$ or $\omega(p)=0$.
In particular, if $\omega(p)=0$, 
then ${\rm rank}(dX)_p=0$ holds.
\item[(ii)]
a point $p\in S_{\infty}(g)$ is a singular point of $X$ 
if and only if $g^2\omega=0$ holds at $p$.
Moreover, in this case, ${\rm rank}(dX)_p=0$ holds.
\end{itemize}
In particular, if $p\in \Sigma$ is a singular point of $X$ 
satisfying ${\rm rank}(dX)_p=1$, then $|g(p)|=1$ holds.
\end{fact}

From the proof of \cite[Theorem 4.1]{Umeda}
and \eqref{eq:bar-z},
the following lemma can be proved easily.

\begin{lemma}\label{lem:nondeg-g}
Let $X : \Sigma \rightarrow L^3$
be a generalized spacelike CMC surface with the Gauss map $g$.
A singular point $p\in S_1(g)$ is non-degenerate
if and only if $dg(p)\neq0$.
In particular, $g$ gives a local diffeomorphism 
on a neighborhood of $p$.
\end{lemma}

\subsection{Proof of Theorem \ref{thm:fold}}

First, we review the definition of fold singular points.

\begin{definition}[Fold singular point]
\label{def:fold}
Let $\Sigma$ be a smooth $2$-manifold
and $X : \Sigma\rightarrow L^3$ a smooth map.
A singular point $p\in \Sigma$ of $X$
is called {\it fold\/} if 
there exist a local coordinate system 
$(U; \varphi)$ around $p\in \Sigma$ 
and a diffeomorphism $\Phi$ of $L^3$
such that $\Phi \circ X \circ \varphi^{-1} = X_{\rm fold}$,
where
$
  X_{\rm fold}(u,v)=(u,\,v^2,\,0).
$
\end{definition}

We use the following fact.

\begin{fact}
For a spacelike immersion 
$X : \Sigma \rightarrow L^3$ 
of an oriented smooth $2$-manifold $\Sigma$,
let $\nu : \Sigma \rightarrow H^2 \subset L^3$ 
be a unit normal vector field,
and $H$ be the mean curvature function.
Then it holds that
\begin{equation}\label{eq:MC-Laplacian}
  \Delta_{ds^2} X = -2H \nu, 
\end{equation}
where $\Delta_{ds^2}$ is the Laplacian of 
the first fundamental form $ds^2$.
\end{fact}

\begin{proof}[{Proof of Theorem \ref{thm:fold}}]
Let $X : \Sigma \rightarrow L^3$
be a generalized spacelike CMC surface with the Gauss map $g$.
Using a suitable homothety,
we may assume that its mean curvature $H$ is $1/2$
without loss of generality.

Assume $p\in \Sigma$ is a fold singular point of $X$.
Then, there exists a local coordinate system
$(U ; u,v)$ around $p$
such that 
$
  X(u,v)=X(u,-v)
$
holds for any $(u,v)\in U$.
Then, we have that
$X_v(u,v) = -X_v(u,-v)$,
and hence $X_v(u,0)=0$ holds.
The singular point set $S(X)$ is given by
$S(X)=\{(u,v) \in U \,;\, v=0\}$.
That is, $\gamma(t)=(t,0)$
gives a singular curve
and $\eta=\partial_v$ gives a null vector field of $X$. 

Since fold singular points are non-degenerate,
$p=(0,0)$ is also non-degenerate and ${\rm rank} (dX)_p=1$ holds.
By Fact \ref{fact:geneCMC-sing}
and Lemma \ref{lem:nondeg-g},
we have $S(X)=\{(u,0)\}\subset S_1(g)$ 
and $g$ gives a local diffeomorphism around $p=(0,0)$.
Hence, if $|g(u_0,v_0)|>1$ holds at a point $(u_0,v_0)\in U_+$,
then $|g(u_0,-v_0)|<1$ holds at $(u_0,-v_0)\in U_-$,
where $U_{\pm}:= \{(u,v)\in U \,;\, v\gtrless 0 \}$.
Since the unit normal $\nu$ of $X$
is given by $\nu=\pi^{-1}\circ g$, i.e.,
\[
  \nu=\frac{1}{|g|^2-1}\left( |g|^2+1, -2\Re g, -2\Im g \right),
\]
we have $\nu(u_0,v_0)\in H^2_+$ and $\nu(u_0,-v_0)\in H^2_-$.

However, by Fact \ref{eq:MC-Laplacian},
it holds that $\nu=-\Delta_{ds^2} X$
on the regular point set $U_+\cup U_-$.
Since 
$
  (\Delta_{ds^2} X) (u,v) = (\Delta_{ds^2} X) (u,-v)
$
holds on $U_+\cup U_-$,
we have $\nu(u_0,v_0)=\nu(u_0,-v_0)\in H^2_+$,
which is a contradiction.
\end{proof}

\begin{remark}
In the above proof of Theorem \ref{thm:fold},
we used the disconnectedness of $H^2$
where the unit normal vector field $\nu$ of a spacelike surface takes values.
In the case of timelike surfaces,
the unit normal vector fields take values in the de Sitter plane 
$S^2_1:=\{p\in L^3\,;\, \inner{p}{p}=1\}$, which is connected.
Hence, 
a proof similar to that of Theorem \ref{thm:fold}
can not be applied directly to 
the timelike case.
\end{remark}

\section{$(2,5)$-cuspidal edge}
\label{sec:criterion}

In this section, we shall prove Theorem \ref{thm:2-5}.
For the proof, we give a criterion for $(2,5)$-cuspidal edges
(Theorem \ref{thm:criterion})
and review the classification of the conjugates of spacelike Delaunay surfaces
(Fact \ref{fact:conjugate-T}, \ref{fact:conjugate-S}, \ref{fact:conjugate-L}).

\subsection{Criterion for $(2,5)$-cuspidal edges}

Let $X : \Sigma\rightarrow \R^3$ be a smooth map
defined on a smooth $2$-manifold $\Sigma$.
A singular point $p\in \Sigma$ of $X$
is called {\it $(2,5)$-cuspidal edge\/} if 
there exist a local coordinate system 
$(U; \varphi)$ around $p\in \Sigma$ 
and a diffeomorphism $\Phi$ of $\R^3$
such that $\Phi \circ X \circ \varphi^{-1} = X_{(2,5)}$,
where
$
  X_{(2,5)}(u,v)=(u,\,v^2,\,v^5)
$
which is called the {\it standard $(2,5)$-cuspidal edge\/}.

\begin{theorem}[A criterion for $(2,5)$-cuspidal edges]
\label{thm:criterion}
Let $U$ be a domain of $\R^2$,
$X : U \rightarrow \R^3$ a frontal,
and $p \in U$ a non-degenerate singular point of the first kind.
Moreover, let
$\gamma(t)$ $(|t|<\varepsilon)$ 
be a singular curve passing through $p=\gamma(0)$
and $\eta(t)$ a null vector field along $\gamma$.
Take smooth vector fields $\xi=\xi(u,v)$ and $\eta=\eta(u,v)$ on $U$
which are extensions of $\gamma'(t)$ and $\eta(t)$, respectively.
Then, $p=\gamma(0)$ is a $(2,5)$-cuspidal edge 
if and only if 
\begin{gather}
  \det(\xi X,\, \eta \eta X,\, \eta \eta \eta X)(\gamma(t))=0
  \qquad (\text{for each} ~ |t|<\varepsilon), \quad and  
  \label{eq:condition-3}\\
  \det(\xi X,\, \tilde\eta \tilde\eta X,\, 3\tilde\eta^5 X - 10 C \tilde\eta^4 X)(p)\neq0
  \label{eq:condition-4}
\end{gather}
hold, where $\eta^k X$ implies $k$-times derivative $\eta \cdots \eta X$, 
$\tilde{\eta}$ is a special null vector field satisfying
\begin{equation}\label{eq:special-null}
  (\xi X \cdot \tilde\eta^2 X)(p)=
  (\xi X \cdot \tilde\eta^3 X)(p)=0
\end{equation}
{\rm (}the dot ``\,$\cdot$\,'' means the Euclidean inner product\/{\rm )},
and $C$ is a constant such that
\begin{equation}\label{eq:constant-C}
  \tilde\eta^3 X(p) = C \tilde\eta^2 X(p)
\end{equation}
holds.
\end{theorem}

By the following lemma,
we have the existence of a null vector field 
satisfying \eqref{eq:special-null}.

\begin{lemma}
\label{lem:special-null-vf}
For a singular point $p$ of the first kind,
there exists a null direction $\tilde\eta$
satisfying the condition \eqref{eq:special-null}.
\end{lemma}

\begin{proof}
Since $p$ is of the first kind,
we may take a coordinate system $(u,v)$ centered at $p$ such that 
$S(X)=\{(u,v) \,;\, u=0\}$ and $dX(\partial_u)=0$ hold.
In this situation, $\xi=\partial_v$ holds.
Then, the vector field
$
  \tilde\eta=\partial_u+ \left( a u+bu^2 \right) \partial_v
$
is the desired null vector field,
where 
\[
  a=-\frac{X_v \cdot X_{uu}}{X_v \cdot X_v}(0,0),\qquad
  b=-\frac{X_v \cdot (X_{uuu}+3a X_{uv})}{2 X_v \cdot X_v}(0,0).
\]
\end{proof}

We shall prove Lemma \ref{lem:indep} in Appendix \ref{app:dependence}.

\begin{lemma}\label{lem:indep}
The conditions \eqref{eq:condition-3} and \eqref{eq:condition-4}
in Theorem \ref{thm:criterion}
are independent of choices of vector fields $\xi$, $\eta$
and coordinate systems of $\R^3$.
\end{lemma}

\begin{proof}[Proof of Theorem \ref{thm:criterion}]
By Lemma \ref{lem:indep}, we have that
a $(2,5)$-cuspidal edge satisfies 
the conditions \eqref{eq:condition-3} and \eqref{eq:condition-4}.
Thus, we here prove the converse.

By Fact \ref{fact:sing-form},
we may write $X(u,v)$ as
$
  X(u,v) = (u,v^2,v^3 h (u,v)).
$
First, we set
\[
  h_0(u):=h(u,0),\qquad
  h_1(u,v):=h(u,v) - h_0(u).
\]
Since $h_1(u,0)=h(u,0) - h_0(u)=0$,
by Fact \ref{fact:division},
there exists a smooth function $\tilde{h}_1(u,v)$ 
defined around the origin
such that $h_1(u,v)=v\tilde{h}_1(u,v)$ holds.
Hence we have
$
  h(u,v) = h_0(u) + v\tilde{h}_1(u,v).
$
Applying Fact \ref{fact:splitting} to
$\tilde{h}_1(u,v)$,
we have $\tilde{h}_1(u,v)=\alpha(u,v^2)+v\beta(u,v^2)$,
and hence
\[
  X(u,v) = \left(u,v^2,v^3 h_0(u) + v^4\alpha(u,v^2)+v^5\beta(u,v^2) \right).
\]
Using the diffeomorphism
$
  \Phi_1 : (x,y,z) \mapsto (x,y,z - y^2 \alpha(x,y) ),
$
we have
\[
  \Phi_1\circ X(u,v) = \left(u,v^2,v^3 h_0(u) +v^5\beta(u,v^2) \right).
\]
Replace $X(u,v)$ by $\Phi_1\circ X(u,v)$.
Then the singular point set of $X$ is $\{(u,0)\}$
and $X_v=0$.
By Lemma \ref{lem:indep}, 
the conditions \eqref{eq:condition-3}, \eqref{eq:condition-4} 
are independent of the choice of vector fields $(\xi,\eta)$. 
Thus we may put $\xi=\partial_u$, $\eta=\partial_v$. Then,
\begin{equation*}
  \xi X(u,0) =(1,0,0),\quad
  \eta^2 X(u,0) =(0,1,0),\quad
  \eta^3 X(u,0) =(0,0,6h_0(u))
\end{equation*}
holds.
By the condition \eqref{eq:condition-3}, we have
$
 h_0(u)=0.
$
Moreover, 
$\partial_v$ also satisfies \eqref{eq:special-null},
and hence we put $\tilde{\eta}=\partial_v$.
Then, $C=0$ holds, where $C$ 
is the constant as in \eqref{eq:constant-C}.
By the condition \eqref{eq:condition-4}, we have
$
  \beta(0,0)\neq0.
$
Therefore, the map
\[
  \Phi_2 : (x,y,z) \longmapsto (x,y,z/\beta(x,y) )
\]
gives a local diffeomorphism
of $\R^3$ around the origin.
Replacing $X(u,v)$ by $\Phi_2\circ X(u,v)$,
we have
$
  X(u,v) = \left(u,v^2,v^5\right).
$
\end{proof}

\subsection{Conjugate of spacelike Delaunay surfaces}

A generalized spacelike CMC surface is called {\it spacelike Delaunay\/} with axis $\ell$
if it is invariant under the action of the group of motions in $L^3$
which fixes each point of the line $\ell$.
Spacelike Delaunay surfaces are classified in
\cite{HanoNomizu}, \cite{IshiharaHara}, \cite{Sasahara} (see also \cite{H_Delaunay}).
The following fact gives those surfaces with non-empty singular point set.

\begin{fact}
\label{fact:Delaunay}
Let $ X : \Sigma\rightarrow L^3$ be a
spacelike Delaunay surface of mean curvature $H$
such that the singular point set of $X$ is not empty.
If the axis of $X$ is 
\begin{itemize}
\item[$\text{(I)}$]
timelike, there exists a constant $k \,(\neq1)$ 
such that $X$ is congruent to 
\begin{equation}\label{eq:t-Delaunay}
  X(r, t) = \frac1{2H}\left(
    \int_{0}^r \frac{\tau^2+k-1}{\sqrt{\delta(\tau)}} d\tau,\, 
    r \cos (2Ht),\, r \sin (2Ht) 
  \right),
\end{equation}
where 
\begin{equation}\label{eq:t-Delaunay2}
  \delta(r) = \left(r^2+k+1\right)^2-4 k.
\end{equation}
\item[$\text{(II)}$]
spacelike, there exists a constant $k \,(\neq1)$ 
such that $X$ is congruent to 
\begin{equation}\label{eq:s-Delaunay}
  X(r, t) = \frac1{2H}\left(
    r \cosh (2Ht),\, r \sinh (2Ht),\, 
    \int_{0}^r \frac{\tau^2 -k+1}{\sqrt{\delta(\tau)}} d\tau
  \right),
\end{equation}
where 
\begin{equation}\label{eq:s-Delaunay2}
  \delta(r) = \left(r^2-k-1\right)^2-4 k.
\end{equation}
\item[$\text{(III)}$]
lightlike, $X$ is congruent to 
\begin{equation}\label{eq:l-Delaunay}
  X(r, t) = \left(\zeta (r)-r\left(1+ \frac{t^2}{4}\right),\, - r t ,\, \zeta (r)+r\left(1- \frac{t^2}{4}\right)\right),
\end{equation}
where $\zeta(r)$ is one of the followings:
\begin{align}
  \label{eq:l-Delaunay-1}
  \zeta(r) &= \frac1{8H^2} \left(-\frac{r }{1+r^2}+\tan ^{-1}r\right),\\
  \label{eq:l-Delaunay-2}
  \zeta(r) &= \frac1{8H^2} \left(\frac{r }{1-r^2}-\tanh ^{-1}r\right).
\end{align}
\end{itemize}
\end{fact}

\begin{remark}
\label{rem:conelike}
For any spacelike Delaunay surface,
its singular point set consists of conelike singularities \cite{H_Delaunay}.
\begin{figure}[htb]
\begin{center}
 \begin{tabular}{{c@{\hspace{8mm}}c@{\hspace{8mm}}c}}
  \resizebox{2.8cm}{!}{\includegraphics{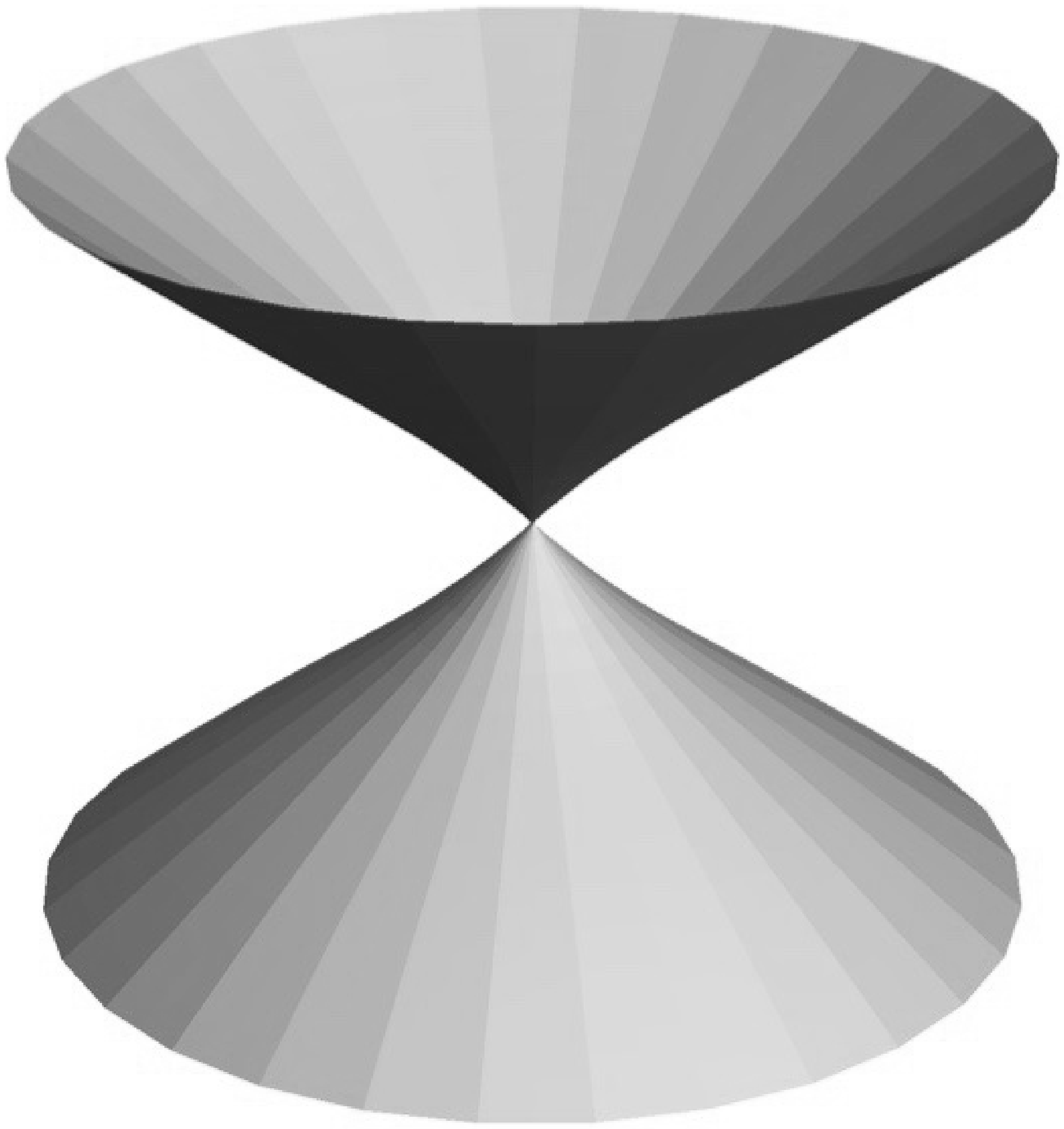}}&
  \resizebox{3.2cm}{!}{\includegraphics{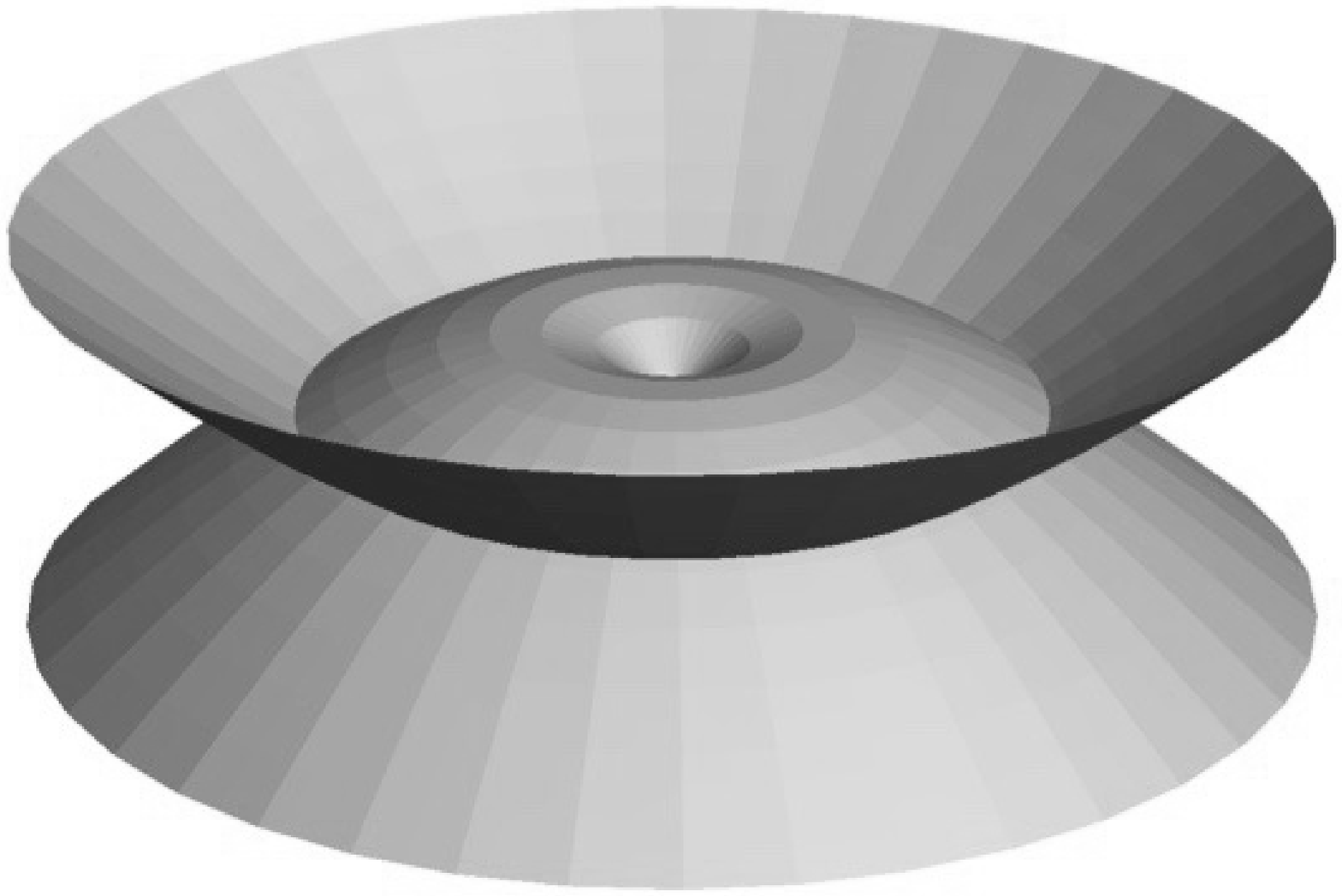}}&
  \resizebox{3.2cm}{!}{\includegraphics{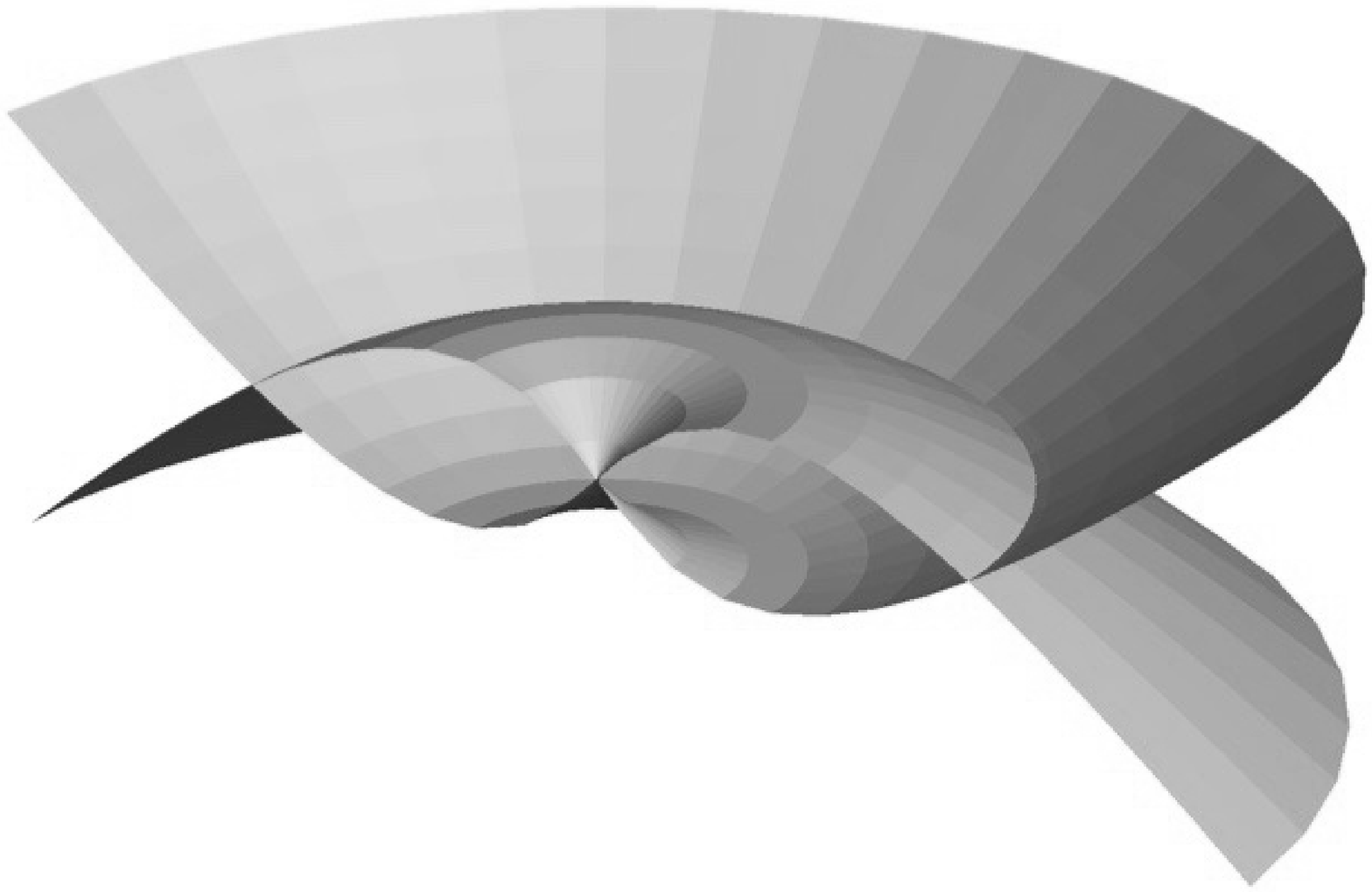}} \\
  {\footnotesize $k=2$ } &
  {\footnotesize $k=0.5$ }&
  {\footnotesize $k=0.5$ (half) }
 \end{tabular}
 \caption{
    Spacelike Delaunay surfaces $X(r,t)$ 
    with timelike axis (cf.\ Fact \ref{fact:Delaunay} (I)). 
    If $k<1$, $X$ has self-intersection.
 }
\label{fig:T-Delaunay}
\end{center}
\end{figure}
\end{remark}

The associate families, in particular, the conjugates of spacelike Delaunay surfaces 
are classified in \cite{H_Delaunay} (see also \cite{Sasahara}).
Set $X_{{T}}(r, t), \, X_{{S}}(r, t),\, X_{{L}}(r, t)$ as 
\begin{align*}
  X_{{T}}(r, t) &= (\lambda+h\, \phi,\, \rho \cos \phi,\, \rho \sin \phi),\\
  X_{{S}}(r,t) &= (\rho \sinh \phi,\, \rho \cosh \phi,\, \lambda+h\, \phi),\\
  X_{{L}}(r, t) &= (\lambda-\rho-\rho \phi^2,\, -2\rho\phi,\, \lambda+\rho-\rho \phi^2)
  +h \left(\frac{\phi^3}{3}+\phi,\, \phi^2,\, \frac{\phi^3}{3}-\phi\right),
\end{align*}
where $h$ is a constant and
$\rho=\rho(r), \, \lambda=\lambda(r)$ 
(resp.\ $\phi=\phi(r,t)$)
are smooth functions of $r$ (resp.\ $(r,t)$).

\begin{fact}[The case of timelike axis]
\label{fact:conjugate-T}
Let $X$ be a spacelike Delaunay surface
whose axis is timelike given by \eqref{eq:t-Delaunay}
and $X^\#$ be its conjugate.
If
\begin{itemize}
\item[$\text{(I-i)}$]
$k>-1$ $(\text{resp.}\ k<-1)$, then $X^\#$ is congruent to
$X_{{T}}(r, t)$ $(\text{resp.}\ X_{{S}}(r, t))$,
where $h= (1-k)/(2H |1+k| )$, and
\begin{equation}\label{eq:t-conj-tDel}
\begin{split}
  &\rho(r) = \frac{\sqrt{\Delta (r)}}{2 H |k+1|},\qquad
  \lambda(r) = 
  \int_{0}^r \frac{\sqrt{ 2|1+k| }\, \tau^4}{H \sqrt{\delta (\tau)} \Delta (\tau)}d\tau,\\
  &\phi(r,t) = 
   {\rm sgn}(k+1)\int_{0}^r  \frac{\sqrt{ 2|1+k| } (1-k) \tau^2}{\sqrt{\delta (\tau)} \Delta (\tau)}  d\tau
  - \sqrt{\frac{|1+k|}{2}}t.
\end{split}
\end{equation}
\item[$\text{(I-ii)}$]
$k=-1$, then $X^\#$ is congruent to $X_{{L}}(r, t)$,
where $h= H$, and $\rho(r) = r/2$,
\begin{equation}\label{eq:l-conj-tDel}
  \lambda(r) = \int_{0}^r \frac{\tau^2(\sqrt{\tau^4+4}+\tau^2)}{4 H^2 \sqrt{\tau^4+4}} d\tau,\qquad
  \phi(r,t) = \int_{0}^r \frac{\sqrt{\tau^4+4}+\tau^2}{2H \sqrt{\tau^4+4}} d\tau + t.
\end{equation}
\end{itemize}
Here, we put $\Delta(r)= 2 (k+1) r^2+(1-k)^2$ and 
$\delta(r)$ is a function given by \eqref{eq:t-Delaunay2}.
\end{fact}

\begin{fact}[The case of spacelike axis]
\label{fact:conjugate-S}
Let $X$ be a spacelike Delaunay surface
whose axis is spacelike given by \eqref{eq:s-Delaunay}
and $X^\#$ be its conjugate.
If
\begin{itemize}
\item[$\text{(II-i)}$]
$k>-1$ $(\text{resp.}\ k<-1)$, then $X^\#$ is congruent to
$X_{{S}}(r, t)$ $(\text{resp.}\ X_{{T}}(r, t))$,
where $h= (1-k)/(2H |1+k| )$, and
\begin{equation}\label{eq:t-conj-sDel}
\begin{split}
  &\rho(r) = \frac{\sqrt{\Delta (r)}}{2 H |k+1|},\qquad
  \lambda(r) = 
  -{\rm sgn}(k+1) \int_{0}^r \frac{\sqrt{ 2|1+k| }\, \tau^4}{H \sqrt{\delta (\tau)} \Delta (\tau)}d\tau,\\
  &\phi(r,t) = 
   \int_{0}^r  \frac{\sqrt{ 2|1+k| } (1-k) \tau^2}{\sqrt{\delta (\tau)} \Delta (\tau)}  d\tau
  -{\rm sgn}(k+1) \sqrt{\frac{|1+k|}{2}}t.
\end{split}
\end{equation}
\item[$\text{(II-ii)}$]
$k=-1$, then $X^\#$ is congruent to $X_{{L}}(r, t)$,
where $h= H$, and $\rho(r) = r/2$,
\begin{equation}\label{eq:l-conj-sDel}
  \lambda(r) = \int_{0}^r \frac{\tau^2(\sqrt{\tau^4+4}+\tau^2)}{4 H^2 \sqrt{\tau^4+4}} d\tau,\qquad
  \phi(r,t) = \int_{0}^r \frac{\sqrt{\tau^4+4}+\tau^2}{2H \sqrt{\tau^4+4}} d\tau + t.
\end{equation}
\end{itemize}
Here, we put $\Delta(r)= -2 (k+1) r^2+(1-k)^2$ and 
$\delta(r)$ is a function given by \eqref{eq:s-Delaunay2}.
\end{fact}

\begin{fact}[The case of lightlike axis]
\label{fact:conjugate-L}
Let $X$ be a spacelike Delaunay surface
whose axis is lightlike given by \eqref{eq:l-Delaunay}
and $X^\#$ be its conjugate.
If $\zeta(r)$ is given by
\begin{itemize}
\item[$\text{(III-i)}$]
\eqref{eq:l-Delaunay-1}, then $X^\#$ is congruent to
$X_{{T}}(r, t)$, where $h= -1/(2H)$ and
\begin{equation}\label{eq:conj-lDel-1}
\begin{split}
  &\rho(r) = \frac{\sqrt{2 r^2+1}}{2 H},\quad
  \lambda(r) = 
  \frac{-\sqrt{2} r+2 \sqrt{2} \tan ^{-1}r-\tan ^{-1}\sqrt{2} r}{2 H},\\
  &\phi(r,t) = 
  \frac{2 H t}{\sqrt{2}}+\sqrt{2} \tan ^{-1}r-\tan ^{-1}\sqrt{2} r.
\end{split}
\end{equation}
\item[$\text{(III-ii)}$]
\eqref{eq:l-Delaunay-2}, then $X^\#$ is congruent to
$X_{{S}}(r, t)$, where $h= 1/(2H)$ and
\begin{equation}\label{eq:conj-lDel-2}
\begin{split}
  &\rho(r) = \frac{\sqrt{1-2 r^2}}{2 H},\quad
  \lambda(r) = 
  \frac{\sqrt{2} r-2 \sqrt{2} \tanh ^{-1}r+\tanh ^{-1}\sqrt{2} r}{2 H},\\
  &\phi(r,t) = 
  \frac{2 H t}{\sqrt{2}}+\sqrt{2} \tanh ^{-1}r-\tanh ^{-1}\sqrt{2} r.
\end{split}
\end{equation}
\end{itemize}
\end{fact}

\begin{proof}[Proof of Theorem \ref{thm:2-5}]
Let $X$ be a spacelike Delaunay surface with conelike singular points
and $X^{\#}$ be its conjugate.
First, we consider the case that
$X$ is given by \eqref{eq:t-Delaunay} and $k>-1$.
Then, by Fact \ref{fact:conjugate-T},
$X^{\#}$ is congruent to $X_{{T}}(r, t)$.
The Euclidean unit normal $\vect{n}$ of $X^{\#}(r,t)$ is given by
\begin{align*}
  \vect{n}(r,t)=
  &\frac{1}{\sqrt{2} \sqrt{\Delta (r)} \sqrt{\delta (r)-(k+1) r^2}}
  \Bigl(
  \sqrt{\delta (r)} \sqrt{\Delta (r)},\\
  &\hspace{3mm}-\sqrt{2} \sqrt{k+1} r^3 \cos \phi (r,t) - (k-1) \sqrt{\delta (r)} \sin \phi (r,t),\\
  &\hspace{8mm}-\sqrt{2} \sqrt{k+1} r^3 \sin \phi (r,t) + (k-1) \sqrt{\delta (r)} \cos \phi (r,t)
  \Bigr),
\end{align*}
where $\Delta(r)$, $\phi(r,t)$ are defined as in Fact \ref{fact:conjugate-T}
and $\delta(r)$ is given by \eqref{eq:t-Delaunay2}.
The signed area density function $\lambda$ is calculated as
\[
  \lambda(r,t)=\det\left((X^{\#})_r,\, (X^{\#})_t,\, \vect{n}\right)
  =\frac{r \sqrt{\delta (r)-(k+1) r^2}}{H \sqrt{k+1} \sqrt{\delta (r)}},
\]
which implies that
the singular point set $S(X^{\#})$ of $X^{\#}(r,t)$ is $S(X^{\#})=\{(r,t)\,;\,r=0\}$.
Since $d\lambda=dr/(H \sqrt{k+1})$ holds on $S(X^{\#})$,
all the singular points of $X^{\#}(r,t)$ are non-degenerate.
Moreover, since the singular curve $\gamma(t)$ and 
the null vector field $\eta(t)$ along $\gamma(t)$ 
are given by $\gamma(t)=(0,t)$, $\eta(t)=\partial_r$, respectively,
all the singular points of $X^{\#}(r,t)$ are of the first kind.
The extensions of $\gamma'$ and $\eta$ are
given by $\xi=\partial_t$ and $\eta=\partial_r$, respectively.
Then we have 
\[
  \det \left(\xi X^{\#},\, \eta\eta X^{\#},\, \eta\eta\eta X^{\#} \right)(\gamma(t))
  =\det \left( (X^{\#})_t,\, (X^{\#})_{rr},\, (X^{\#})_{rrr} \right)(0,t)
  =0
\]
for each $t$, 
and hence the condition \eqref{eq:condition-3} holds.
By Lemma \ref{lem:special-null-vf},
the null vector field $\tilde\eta$ satisfying \eqref{eq:special-null}
is calculated as
\[
  \tilde\eta=\partial_r+\frac{r^2}{H (k-1) |k-1|}\partial_t.
\]
Since $\tilde\eta^3 X^{\#}(0,t)=0$, 
the constant $C$ as in \eqref{eq:constant-C} is $0$.
Then we have 
\[
  \det \left(\xi X^{\#},\, \tilde\eta^2 X^{\#},\, \tilde\eta^5 X^{\#} \right)(\gamma(t))
  =-\frac{72}{H^2 |k-1|^3}\neq0
\]
for each $t$, and hence the condition \eqref{eq:condition-4} holds.
Therefore, Theorem \ref{thm:criterion} yields that 
all the singular points of $X^{\#}(r,t)$ are $(2,5)$-cuspidal edges.
In the case that
$X$ is given by \eqref{eq:s-Delaunay}, \eqref{eq:l-Delaunay}
or \eqref{eq:t-Delaunay} with $k\leq -1$, 
we can prove the desired result in a similar way.
\end{proof}

\begin{remark}
\label{rem:maxface_25}
We should remark that maxfaces do not admit any $(2,5)$-cuspidal edges.
In fact, if we assume that a maxface admit a $(2,5)$-cuspidal edge $p$,
then by \cite[Lemma 3.3]{UY_maxface} and \cite[Lemma 2.17]{FKKRSUYY_Okayama}, 
one can show that $p$ satisfies the condition of folds, which is a contradiction
(see also \cite[Definition 2.13]{FKKRSUYY_Okayama}).
Therefore, 
the singularity types of generalized spacelike CMC surfaces 
are different from those of maxfaces.
\end{remark}

\appendix

\section{Proof of Lemma \ref{lem:indep}}
\label{app:dependence}

Let $\xi$, $\eta$ be smooth vector fields 
which satisfy the assumptions in Theorem \ref{thm:criterion}.
That is, let $X : U \rightarrow \R^3$ be a frontal,
$p \in U$ a singular point of the first kind,
$\gamma(t)$ $(|t|<\varepsilon)$ 
be a singular curve passing through $p=\gamma(0)$
and
$\xi=\xi(u,v)$ and $\eta=\eta(u,v)$ 
are smooth vector fields on $U$
which are extensions of the singular direction $\gamma'(t)$ 
and the null vector field $\eta(t)$,
respectively.

We shall prove Lemma \ref{lem:indep} 
in the following two steps (Step \ref{step:1}, Step \ref{step:2}).

\begin{step}\label{step:1}
The conditions \eqref{eq:condition-3} and \eqref{eq:condition-4}
in Theorem \ref{thm:criterion}
are independent of choices of vector fields $\xi$, $\eta$.
\end{step}
\begin{proof}
If $\bar\xi$, $\bar\eta$ are also vector fields 
satisfying the assumptions in Theorem \ref{thm:criterion},
they can be expressed as a linear combination
\begin{equation}\label{eq:abcd}
  \bar\xi=a_1(u,v)\xi+a_2(u,v)\eta,\qquad
  \bar\eta=b_1(u,v)\xi+b_2(u,v)\eta,
\end{equation}
where $a_j$, $b_j$ $(j=1,2)$ are smooth functions satisfying
\[
  a_2(u,v)=b_1(u,v)=0 
\]
on the singular point set $S(X)$, and $a_1(u,v)$, $b_2(u,v)$ never vanish on $S(X)$.
Then, it holds that
\[
  \bar\xi X(p)=a_1(p)\xi X(p),\qquad
  \bar\eta\bar\eta X(p)=
  b_2(\eta b_1\xi X+b_2\eta\eta X)(p).
\]

First, we shall prove that the condition 
\eqref{eq:condition-3} is independent of the choice of $\xi$ and $\eta$.
It suffices to show that
$\det(\bar\xi X,\bar\eta\bar\eta X,\bar\eta\bar\eta\bar\eta X)(p)$
is a non-zero constant multiple of
$\det(\xi X,\eta\eta X,\eta\eta\eta X)(p)$.
Since we want to calculate 
$\det(\bar\xi X,\bar\eta\bar\eta X,\bar\eta\bar\eta\bar\eta X)$,
we shall ignore the terms of $\xi X(p)$, $\eta\eta X(p)$
appearing in $\bar\eta\bar\eta\bar\eta X(p)$.
As $\xi \eta-\eta\xi$ tangent to $S(X)$
and the image of $dX$ is spanned by $\xi X$ on $S(X)$, 
$(\xi \eta-\eta\xi) X$ is parallel to $\xi X$ on $S(X)$.
Moreover, since $\eta X=0$ on $S(X)$,
we have that $\xi\eta X=0$ on $S(X)$.
Therefore, we can also ignore $\eta\xi X(p)$.
In this situation, 
$\bar\eta\bar\eta\bar\eta X(p)=(b_2)^3\eta\eta\eta X(p)$ holds, 
and hence we have that 
the condition \eqref{eq:condition-3} is 
independent of choices of vector fields.

With respect to the condition \eqref{eq:condition-4},
let $\xi$, $\eta$ be vector fields
which satisfy the assumptions in Theorem \ref{thm:criterion}.
Moreover, we assume that the condition \eqref{eq:special-null},
that is $\eta\eta\eta X(p)={C} \eta\eta X(p)$ holds.
If $\bar\xi$, $\bar\eta$ are also vector fields
satisfying these assumptions,
they can be expressed as a linear combination
in \eqref{eq:abcd}.
Then, we have 
\[
  \eta b_1(p)=\eta\eta b_1(p)=0.
\]
Under these assumptions, it holds that
\[
  \bar\eta\bar\eta X(p)=(b_2)^2\eta\eta X(p),
  \qquad
  \bar\eta\bar\eta\bar\eta X(p)
  =(b_2)^2(3\eta b_2+Cb_2)(p)\eta\eta X(p).
\]
Hence the constant $\bar{C}$ satisfying
$\bar\eta\bar\eta\bar\eta f(p) =\bar{C} \bar\eta\bar\eta f(p)$
is given by  
$\bar{C}=(3\eta b_2+Cb_2)(p).$
Now we shall show that
$\det(\bar\xi X,\,\bar\eta\bar\eta X,\,3\bar\eta^5X-10\bar{C}\bar\eta^4 X)(p)$
is a non-zero constant multiple of 
$\det(\xi X,\,\eta\eta X,\, 3\eta^5X-10C\eta^4 X)(p).$
As in the argument above,
we shall ignore the terms which are parallel to
$\xi X(p),\eta\eta X(p)$.
Then, we have
\[
  \bar\eta^4 X(p)= (b_2)^4\eta^4 X(p),\qquad
  \bar\eta^5 X(p)=
  (b_2)^4\left\{
  10(\eta b_2)(\eta^4 X)
  +b_2\,\eta^5 X\right\}(p).
\]
Therefore, it holds that 
$
  (3\bar\eta^5X-10\bar{C}\bar\eta^4 X)(p)
  =(b_2)^5(3\eta^5X-10{C}\eta^4X)(p),
$
and hence we have the conclusion.
\end{proof}

\begin{step}\label{step:2}
The conditions \eqref{eq:condition-3} and \eqref{eq:condition-4}
in Theorem \ref{thm:criterion}
are independent of a choice of coordinate systems of $\R^3$.
\end{step}
\begin{proof}
Without loss of generality,
we may assume that $p=(0,0)$ and $X(p)=(0,0,0)$.
Let $\Phi=(\Phi^1,\Phi^2,\Phi^3)$ be a diffeomorphism
such that $\Phi(0,0,0)=(0,0,0)$ and
$(x_1,x_2,x_3)$ a coordinate system of $\R^3$.
In the following, we denote by $(\Phi^k)_{k=1,2,3}$ the point 
$(\Phi^1,\Phi^2,\Phi^3)$.

Since $\eta$ is a null direction, we have
{\allowdisplaybreaks
\begin{align*}
  \xi (\Phi\circ X)
  &= \left(\sum_{i=1}^3\Phi^k_{x_i}\xi X_i\right)_{k=1,2,3}
  =d\Phi(\xi X),
  \\
  \eta (\Phi\circ X)
  &=\left(\sum_{i=1}^3\Phi^k_{x_i}\eta X_i\right)_{k=1,2,3}
  =d\Phi(\eta X),
  \\
  \eta^2 (\Phi\circ X)
  &=\left(\sum_{i=1}^3\Phi^k_{x_i}\eta^2 X_i\right)_{k=1,2,3}
  =d\Phi(\eta^2 X),
  \\
  \eta^3 (\Phi\circ X)
  &=\left(\sum_{i=1}^3\Phi^k_{x_i}\eta^3 X_i\right)_{k=1,2,3}
  =d\Phi(\eta^3 X),
  \\
  \eta^4 (\Phi\circ X)
  &=\left(\sum_{i,j=1}^3 3\Phi^k_{x_ix_j}\eta^2 X_i\eta^2 X_j
     +\sum_{i=1}^3\Phi^k_{x_i}\eta^4 X_i\right)_{k=1,2,3},
  \\
  \eta^5 (\Phi\circ X)
  &=
  \left(\sum_{i,j=1}^3 10\Phi^k_{x_ix_j}\eta^3 X_i\eta^2 X_j
     +\sum_{i=1}^3\Phi^k_{x_i}\eta^5 X_i\right)_{k=1,2,3}.
\end{align*}}
Here, we regard $d\Phi:T\R^3\to T\R^3$
as a $GL(3,\R)$-valued map on $\R^3$.
These yield that the condition \eqref{eq:condition-3}
is independent of a choice of coordinate system of $\R^3$.
Moreover, we have that the condition \eqref{eq:special-null}
is also independent of a choice of coordinate system of $\R^3$.

From now on, we assume that $\tilde\eta$ satisfies the condition \eqref{eq:special-null}.
Since at the origin
\[
  \tilde\eta^3 X=C\tilde\eta^2 X,\qquad
  \tilde\eta^3(\Phi\circ X)=\hat{C}\tilde\eta^2(\Phi\circ X),
\]
hold, we have
$\tilde\eta^2(\Phi\circ X)=d\Phi(\tilde\eta^2 X)$, 
$\tilde\eta^3(\Phi\circ X)=d\Phi(\tilde\eta^3 X)$,
and hence
\[
  \tilde\eta^3(\Phi\circ X) - \hat{C}\tilde\eta^2(\Phi\circ X)
  =d\Phi(\tilde\eta^3 X-\hat{C}\tilde\eta^2 X)
  =d\Phi(C\tilde\eta^2 X-\hat{C}\tilde\eta^2 X)
\]
at the origin, which imply $\hat{C}=C$.
{\allowdisplaybreaks
Now, at the origin, it holds that
\begin{align*}
&3\tilde\eta^5(\Phi\circ X) - 10{C}\tilde\eta^4(\Phi\circ X)\\
&\hspace{10mm}=
     3\Bigg(
        \sum_{i,j=1}^3 10\Phi^k_{x_ix_j}\eta^3 X_i\eta^2 X_j
        +\sum_{i=1}^3\Phi^k_{x_i}\eta^5 X_i
     \Bigg)_{k=1,2,3}\\
&\hspace{30mm}
    -10{C}\Bigg(
       \sum_{i,j=1}^3 3\Phi^k_{x_ix_j}\eta^2 X_i\eta^2 X_j
      +\sum_{i=1}^3\Phi^k_{x_i}\eta^4 X_i
    \Bigg)_{k=1,2,3}\\
&\hspace{10mm}=
   \Bigg(
     30\sum_{i,j=1}^3 \Phi^k_{x_ix_j}\eta^2 X_j
     \big(\eta^3 X_i-{C}\eta^2 X_i\big)
   \Bigg)_{k=1,2,3}\\
&\hspace{50mm}
   +\Bigg(
      \sum_{i=1}^3\Phi^k_{x_i}\big(3\eta^5 X_i
     -10{C}\eta^4 X_i\big)
   \Bigg)_{k=1,2,3}\\
&\hspace{10mm}=d\Phi\big(3\eta^5 X_i-10{C}\eta^4 X_i\big).
\end{align*}}
Thus,
we have the condition \eqref{eq:condition-4}
is independent of a choice of coordinate system of $\R^3$.
\end{proof}

\begin{acknowledgements}
The authors thank Professors Masaaki Umehara and Kotaro Yamada 
for valuable comments.
The first and the second authors are partially supported by 
Grant-in-Aid for Challenging Exploratory Research No. 26610016
of the Japan Society for the Promotion of Science. 
The second author is partially supported 
by Grant-in-Aid for Scientific Research (B) No. 25287012 
and the third author by (C) No. 26400087 
from Japan Society for the Promotion of Science.
\end{acknowledgements}


\end{document}